\documentclass[conference] {IEEEtran}
\usepackage{amsmath,amssymb, amsfonts, amsbsy, latexsym, color,bbm,amsthm}
\usepackage{graphicx}
\usepackage{fullpage}

\theoremstyle{plain}
\newtheorem{thm}{Theorem}[section]

\newtheorem{prop}[thm]{Proposition}

\theoremstyle{definition}
\newtheorem{defn}[thm]{Definition}
\theoremstyle{remark}

\newcommand{\cG}{\mathcal{G}}

\newcommand{\cI}{\mathcal{I}}

\newcommand{\bZ}{\mathbb{Z}}

\newcommand{\bC}{\mathbb{C}}

\newcommand{\tr}{\operatorname{tr}}

\newcommand{\SYM}{\mathrm{Sym}}

\newcommand{\ip}[2]{\left\langle#1,#2\right\rangle}
\newcommand{\norm}[1]{\left\Vert#1\right\Vert}
\newcommand{\abs}[1]{\left\vert#1\right\vert}
\newcommand{\absip}[2]{\left\vert\langle#1,#2\rangle\right\vert}

\newcommand{\beq}{\begin{equation}}
\newcommand{\eeq}{\end{equation}}

\DeclareMathOperator*{\TP}{TP}
\DeclareMathOperator*{\diag}{diag}

\begin{document}

\title{$2$- and $3$-Covariant Equiangular Tight Frames}

\author{\IEEEauthorblockN{Emily J.\ King}
\IEEEauthorblockA{Center for Industrial Mathematics\\
University of Bremen, Bremen, Germany\\
Email: king@math.uni-bremen.de}}

\maketitle

\begin{abstract}
Equiangular tight frames (ETFs) are configurations of vectors which are optimally geometrically spread apart and provide resolutions of the identity. Many known constructions of ETFs are group covariant, meaning they result from the action of a group on a vector, like all known constructions of symmetric, informationally complete, positive operator-valued measures. In this short article, some results characterizing the transitivity of the symmetry groups of ETFs will be presented as well as a proof that an infinite class of so-called Gabor-Steiner ETFs are roux lines, where roux lines are a generalization of doubly transitive lines. \end{abstract}

\section{Introduction}
Frames are generalizations of orthonormal bases which have applications in signal processing, quantization, coding theory, and more~\cite{frame_book,Waldron18}.  Equiangular tight frames are the closest analog to orthonormal bases in a redundant setting and are known to give representations of data that are optimally robust to erasures and noise~\cite{StH03}.  Many equiangular tight frames of interest are generated by group actions.  Understanding the higher order symmetries of a equiangular tight frame yields information about the structure of the frame and when such equiangular tight frames may exist.  In Section~\ref{sec:gpcov} double and triple covariance of equiangular tight frames are characterized -- completely so in the latter case --, generalizing results in the quantum information literature (in particular~\cite{Zhu15}) about symmetric, informationally complete, positive operator-valued measures, which are a specific class of equiangular tight frames.  In Section~\ref{sec:drackn}, the covariance properties of so-called Gabor-Steiner equiangular tight frames~\cite{BoKi18} are explored.  In particular, a class of Gabor-Steiner equiangular tight frames are shown to be roux lines, a generalization of both abelian distance-regular antipodal covers of the complete graph and doubly transitive equiangular tight frames~\cite{IvMi18}.

\section{Equiangular Tight Frames and Group Covariance}\label{sec:gpcov}

Equiangular tight frames emulate the algebraic and geometric properties of orthonormal bases but may be redundant.  
\begin{defn}
Let $\Phi = \{\varphi_j\}_{j=1}^n \subset \bC^d$. Then $\Phi$ is an \emph{equiangular tight frame (ETF)} if the following hold: 
\begin{enumerate}
\item for all $x \in \bC^d$, $x= \frac{d}{n} \sum_{j=1}^n \langle x, \varphi_j \rangle \varphi_j,$
\item $\norm{\varphi_j} = 1$ for all $j \in \{1, \hdots, n\}$, and
\item there exists $\alpha \geq 0$ such that $\absip{\varphi_j}{\varphi_k}=\alpha$  for all $j \neq k$.
\end{enumerate}
\end{defn}
It ends up that traits 1) and 3) imply that the absolute values of the inner products are optimally small; that is, the vectors are as geometrically as spread as possible.
\begin{thm}\cite{StH03,LS1973} \label{thm:Welch}
Let $\Phi = \{\varphi_j\}_{j=1}^n \subset \bC$ be a set of unit vectors.  Then
\begin{equation}\label{eqn:Welch}
\max_{j \neq k} \absip{\varphi_j}{\varphi_k} \geq \sqrt{\frac{n-d}{d(n-1)}}.
\end{equation}
The bound in~\eqref{eqn:Welch} is saturated if and only if $\Phi$ is an equiangular tight frame.  Further, the bound in~\eqref{eqn:Welch} may only be saturated if $n \leq d^2$.
\end{thm}
The bound in~\eqref{eqn:Welch} is called the \emph{Welch bound}, and the bound $n \leq d^2$ is \emph{Gerzon's bound}.  
It is conjectured that there is always an ETF of $d^2$ vectors in $\bC^d$ \cite{Zauner1999}; this is called \emph{Zauner's conjecture}.
This conjecture originally arose in quantum information theory, where such maximal ETFs are called \emph{symmetric, informationally complete, positive operator-valued measures (SICs)}. A stronger variant of Zauner's conjecture is that such SICs may always be generated as the orbit of a single vector under a projective unitary representation of $\bZ_p \times \bZ_p$ related to a finite Weyl-Heisenberg group. In general, we call any ETF which is formed as the orbit of a single vector under a (projective) unitary representation \emph{group covariant}.  By definition, the action of the group on a group covariant ETF is \emph{transitive}; that is, given any two vectors in the ETF, there is a unitary mapping parameterized by a group element that maps one to the other.  Unitary transformations leave the quantum state space invariant, so it is of interest to ask when permutations of the associated rank-one projections of a SIC~\cite{Zhu15} or other group covariant ETF can be realized by such operators. 
\begin{defn}
For an ETF $\Phi$ parameterized by a group we denote by $G$ the group of unitary operators for which the ETF is invariant.  That is, for all $U \in G$ and $\varphi_j \in \Phi$, $U \varphi_j \varphi_j^\ast U^\ast = \varphi_{\sigma(j)} \varphi_{\sigma(j)}^\ast$ for some permutation $\sigma$ of the group parameterization. The \emph{symmetry group} of $\Phi$ is $\overline{G}=G/S^1$, that is, $G$ up to multiplication by universal phase factor. If the symmetry group maps every ordered $k$-tuple of distinct elements to every ordered $k$-tuple of distinct elements (i.e., is $k$-transitive), then we call the ETF \emph{$k$-covariant}. 
\end{defn}
The symmetry group yields important structural information about the ETF (see, e.g.,~\cite{AFF11}). Doubly and triply covariant SICs were completely classified in~\cite{Zhu15}.
\begin{thm}\cite{Zhu15} \label{thm:3SIC}
There are no triply covariant SICs. Up to equivalence, the doubly covariant SICs are
\begin{itemize}
\item SICs in $\bC^2$,
\item the Hesse SIC (a certain type of SIC in $\bC^3$ with many linear dependencies, also the Gabor-Steiner-ETF over $\bZ_3$~\cite{Hugh07,DBBA2013,BoKi18}), and
\item Hoggar's lines (a sporadic SIC formed by the Weyl-Heisenberg group over $\bZ_2 \times \bZ_2 \times \bZ_2$ rather than a cyclic group \cite{Hog98}).
\end{itemize}
\end{thm}
In order to characterize triply covariant ETFs, we will make use of so-called triple products. 
\begin{defn} \cite{AFF11,ChWa16,FJKM17}
Let $\Phi = \{\varphi_j\}_{j=1}^n$ be an ETF for $\bC^d$. For $j, k, \ell \in \{1, \hdots, n\}$ we define the \emph{triple product} to be $\TP(j,k,\ell) = \ip{\varphi_j}{\varphi_k}\ip{\varphi_k}{\varphi_\ell}\ip{\varphi_\ell}{\varphi_j}.$
If all of the triple products of distinct $j,k,\ell$ are real and negative, then $\Phi$ is a \emph{simplex}.
\end{defn}
\begin{thm}
If an ETF $\Phi$ of $n$ vectors in $\bC^d$ is triply covariant, $d=1$, $n=d$, or $n=d+1$.  That is, the only non-trivial triply covariant ETFs are orthonormal bases and simplices.
\end{thm}
\begin{proof} We assume that $n \geq 3$ and generalize the proof of Theorem~\ref{thm:3SIC} found in \cite[Lemma 5]{Zhu15}.  Let $\Phi=\{ \varphi_j\}_{j=1}^n$ be a triply covariant ETF for $\bC^d$. Then all of the triple products (of distinct vectors) must be equal.  Since for all $j \neq k$, $\TP(j,k,\ell) = \overline{\TP(k,j,\ell)}$,
all of the triple products (of distinct vectors) must be real.  We fix $j \neq k$ and note that by the Welch bound (Theorem~\ref{thm:Welch})
\begin{align}
\lefteqn{\sum_{\ell=1}^n \TP(j,k,\ell)= \left( \sum_{\ell \notin\{j,k\}} +  \sum_{\ell\in\{j,k\}} \right) \TP(j,k,\ell)}\nonumber\\
&= \pm(n-2) \left(\sqrt{\frac{n-d}{d(n-1)}}\right)^3 +2 \left(\sqrt{\frac{n-d}{d(n-1)}}\right)^2, \label{eq:tripcov1}
\end{align}
and
\begin{align}
\lefteqn{\sum_{\ell=1}^n \TP(j,k,\ell) = \sum_{\ell=1}^n\overline{ \TP(k,j,\ell)}}\nonumber\\
&=\sum_{\ell=1}^n \overline{\tr(\varphi_j^\ast\varphi_k \varphi_k^\ast\varphi_\ell\varphi_\ell^\ast\varphi_j ) }=\overline{\tr(\varphi_j\varphi_j^\ast\varphi_k \varphi_k^\ast \sum_{\ell=1}^n \varphi_\ell\varphi_\ell^\ast ) }\nonumber\\
&= \frac{n}{d} \absip{\varphi_j}{\varphi_k}^2 = \frac{n}{d} \left(\sqrt{\frac{n-d}{d(n-1)}}\right)^2.\label{eq:tripcov2}
\end{align}
We set \eqref{eq:tripcov1} and \eqref{eq:tripcov2} to be equal. Then either the ETF is an orthonormal basis or one may divide each side by $(n-d)/((d(n-1))$.  In the latter case, one obtains the equation $0 = (1-d)n^2(n-d-1)$,
yielding the nonsense solution $n=0$, the trivial solution $d=1$, and the solution $n=d+1$, where all ETFs of $d+1$ vectors in $\bC^d$ are simplices~\cite{FJKM17}.
\end{proof}

For doubly transitive ETFs, we have the following result, generalizing Lemma 8 in \cite{Zhu15}.
\begin{prop}
Let $\Phi = \{\varphi_j \}_{j=1}^n$ be a doubly transitive ETF for $\bC^d$ with $n>d$. Then for all $j \neq k \neq \ell$, there exists a $2n$th root of unity $\zeta_{j,k,\ell}$ such that
\[
\TP(j,k,\ell) = \zeta_{j,k,\ell} \left(\frac{n-d}{d(n-1)}\right)^{3/2}.
\]
\end{prop} 
\begin{proof}
For $j, k, \ell$, we set 
\[
\widetilde{\TP}(j,k,\ell) =\TP(j,k,\ell) / \abs{\TP(j,k,\ell)}.
\]
Fix $j,k \in \{1,\hdots, n\}$ with $j \neq k$.  The double transitivity yields that the multisets
\begin{align*}
&\left\{\widetilde{\TP}(m,j,k) : m \in \{1, \hdots, n\} \right\}\enskip \textrm{and}\\
&\left\{\widetilde{\TP}(m,k,j) : m \in \{1, \hdots, n\} \right\}
\end{align*}
are identical.  However, the sesquilinearity of the inner product yields that the elements of the two multisets are conjugates of each other.  Since they are conjugate invariant,
\[
\prod_{m=1}^n \widetilde{\TP}(m,j,k) = \pm 1
\]
where the sign is independent of choice of (distinct) $j$ and $k$.  We further note that for any $j,k,\ell,m$,
\begin{equation}\label{eqn:lemmTP}
\widetilde{\TP}(j,k,\ell) = \widetilde{\TP}(m,j,k)  \widetilde{\TP}(m,k,\ell) \widetilde{\TP}(m,\ell,j). 
\end{equation}
By taking the product of~\eqref{eqn:lemmTP} over all $m \in \{1, \hdots, n\}$, we obtain $\widetilde{\TP}(j,k,\ell)^n = \pm 1$.  Since for distinct $j,k,\ell$, 
\[
\widetilde{\TP}(j,k,\ell) = \left(\frac{d(n-1)}{n-d}\right)^{3/2}\TP(j,k,\ell),
\]
the lemma follows.
\end{proof}

\section{Roux Lines}\label{sec:drackn}
Equivalence classes of real equiangular tight frames are known to be in one-to-one correspondence to combinatorial objects known as regular two-graphs~\cite{Sei76,HoPa04}. The correspondence is related to the fact that the inner products of equiangular vectors in real Euclidean space take one of two values based on their sign and these values can be thought of determining adjacency.  Since the inner products of equiangular vectors in complex space could have infinitely many phases, the situation in complex space is more complicated.  In~\cite{IvMi18}, a complex analogue of regular two-graphs the authors call \emph{roux} is developed by requiring that the Gram matrix of the vectors satisfy certain axioms concerning association schemes. Unlike in the real case, not all complex equiangular tight frames yield roux lines. All doubly transitive ETFs are roux lines. We will prove that certain Gabor-Steiner equiangular tight frames correspond to roux lines.

We begin by defining the class of ETFs, Gabor-Steiner ETFs~\cite{BoKi18}, which we would like to analyze. Like SICs, these are generated by the orbit of a single vector under a projective unitary representation of a Weyl-Heisenberg-like group; however, except for the case $m=3$,  Gabor-Steiner ETFs are not SICs. 
\begin{defn}
Let $m \geq 2$ be an integer and $\zeta_m \in \bC$ a primitive $m$th root of unity. We denote the $m \times m$ identity matrix by $I_m$. Furthermore, we define the \emph{(cyclic) translation} $T_m$ and \emph{modulation} $M_m$ operators as
\[
T_m = (\operatorname{circ}(0,1, 0, \hdots, 0),\, M_m =\diag( 1 ,\hdots , \zeta_m^{m-1}).
\]
Further, if $m=(m_0, m_1, \hdots, m_s)$ is a vector of integers $\geq 3$, the group of translations over $\bigoplus_{\ell=0}^s \bZ_{m_\ell}$ is
\[
\left\{T_m^{(k)} :=  \bigotimes_{\ell=0}^s T_{m_\ell}^{k_\ell}: k=(k_0, \hdots k_s) \in \bigoplus_{\ell=0}^s \bZ_{m_\ell}\right\},
\]
where $\otimes$ is the Kronecker product.  Similarly, the group of modulations is
\[
\left\{M_m^{(\kappa)}:= \bigotimes_{\ell=0}^s M_{m_\ell}^{\kappa_\ell}: \kappa=(\kappa_0, \hdots \kappa_s) \in \bigoplus_{\ell=0}^s \bZ_{m_\ell}\right\}.
\]
If further each $m_\ell$ is odd, we define the projective unitary representation $\pi$ on $\bigoplus_{\ell=0}^s \bZ_{m_\ell} \times \bigoplus_{\ell=0}^s \bZ_{m_\ell}$ as 
\[
\pi(k, \kappa) = I_{(\abs{m}-1)/2} \otimes \left( M_{m}^{(\kappa)}T^{(k)}_{m}\right).
\] 
\end{defn}
\begin{thm}\cite{BoKi18}
Let $m=(m_0, \hdots, m_s)$ be a vector of odd integers $\geq 3$ and set $\abs{m} = \prod_{\ell=0}^s m_\ell$.

Let
\[
\mathcal{I}=\left\{ (0, \hdots, 0) , \hdots, ((m_0-1)/2, \hdots, (m_s-3)/2)\right\},
\]
which is the set of the first $(\abs{m}-1)/2$ elements of $\bigoplus_{\ell=0}^s \bZ_{m_\ell}$, ordered lexicographically. For $i \in \cI$ define
\[
\left(\left( \phi_{i}\right)_{j}\right)_j=\left(\left\{ \begin{array}{lr} 1; & j=i\\ -1; &j= m-i-\mathbbm{1}\\ 0; & \textrm{o.w.}\end{array}\right.\right)_j\in \bC^{\abs{m}},
\]
where $\mathbbm{1}$ is the all-ones vector of length $\abs{m}$, and $\psi$ to be the block vector in $\bC^{\abs{m}(\abs{m}-1)/2}$ consisting of the $\phi_i$ stacked vertically. We finally define $\cG(m)$ to be the orbit of $\psi$ under $\pi(\bigoplus_{\ell=0}^s \bZ_{m_\ell} \times \bigoplus_{\ell=0}^s \bZ_{m_\ell})$.   Then $\cG(m)$ is an ETF called a {\it Gabor-Steiner ETF}; 
\end{thm}
Gabor-Steiner ETFs span the same set of lines as the ETFs in~\cite{BoEl10a,IJM17}. We will make use of signature matrices and their characterization of ETFs (see, e.g., \cite{LS1973,HoPa04}).
\begin{defn}
Let $\Phi$ be an ETF of vectors of norm $\nu$ and absolute inner product value $\alpha>0$. The \emph{signature matrix} $S$ (also called \emph{Seidel matrix}) of $\Phi$ is defined to be $S = (\Phi^\ast \Phi- \nu^2 I)/\alpha$.
If $\overline{\Phi}$ is switching equivalent to $\Phi$ (i.e., spans the same set of lines) and has signature matrix $\overline{S}$, where the entries in the first row and column with the exception of the diagonal element are equal to one, then $\overline{S}$ is a \emph{normalized signature matrix} of $\Phi$.
\end{defn}
\begin{prop}\label{prop:sign}
Let $\Phi$ be an equiangular tight frame of $n$ vectors in $\bC^d$ with signature matrix $S$.  Then the following hold true.
\begin{enumerate}
\item[(i)] $S \in \SYM_n(\bC)$;
\item[(ii)] The diagonal entries of $S$ are all zero;
\item[(iii)] The off-diagonal entires of $S$ are unimodular;
\item[(iv)] $S$ has two unique eigenvalues; and
\item[(v)] The larger eigenvalue of $S$ has multiplicity $d$.
\end{enumerate}
Further, if a matrix $S$ satisfies (i)--(v), then there exists an equiangular tight frame $\Phi$ of $n$ vectors of norm $\nu$ and absolute inner product value $\alpha$ in $\bC^d$ such that $S = (\Phi^\ast \Phi- \nu^2 I)/\alpha$.
\end{prop}

\begin{prop}\label{prop:sign}
Let $m=(m_0, m_1, \hdots, m_s)$ be a vector of odd integers $\geq 3$.  Define
\[
\overline{S} = \left( s_{(k,\kappa),(\tilde{k},\tilde{\kappa})}\right)_{(\tilde{k},\tilde{\kappa}),(k,\kappa) \in \left(\bigoplus_{\ell=0}^s \bZ_{m_\ell} \times \bigoplus_{\ell=0}^s \bZ_{m_\ell}\right)},
\]
where 
\[
s_{(k,\kappa),(\tilde{k},\tilde{\kappa})} = -\prod_{\ell=0}^s \zeta_{m_\ell}^{(\kappa_\ell \tilde{k}_\ell - \tilde{\kappa}_\ell k_\ell)/2}
\]
when $(\tilde{k},\tilde{\kappa})$, $(k,\kappa)$, and $(0,0)$ are distinct and $s_{(k,\kappa),(\tilde{k},\tilde{\kappa})} = 1-\delta_{(k,\kappa),(\tilde{k},\tilde{\kappa})}$ otherwise.
Then $\overline{S}$ is a normalized signature matrix of $\cG(m)$.
\end{prop}
\begin{proof}
Let $\Phi = \cG(m)$. It follows from \cite[Lemma 5.1]{BoKi18} that 
\begin{align*}
\lefteqn{S = \Phi^\ast \Phi - (\abs{m}-1)I} \\
&=\left(\!\!\begin{array}{lr}-\prod_{\ell=0}^s \zeta_{m_\ell}^{(\kappa_\ell-\tilde{\kappa}_\ell)(\tilde{k}_\ell+k_\ell-1)/2}, \!& \!(\tilde{k},\tilde{\kappa}) \neq (k,\kappa) \\ 0, \!&\!(\tilde{k},\tilde{\kappa})=(k,\kappa) \end{array} \!\!\right).
\end{align*}
We form a related equiangular tight frame $\overline{\Phi}$ by multiplying each $\varphi_{k,\kappa}$ by $\prod_{\ell=0}^s \zeta_{m_\ell}^{-\kappa_\ell(k_\ell-1)/2}$ and additionally  $\varphi_{0,0}$ by $-1$. Since each vector is multiplied by a unimodular, $\overline{\Phi}$ is switching equivalent to $\Phi$.  The signature matrix $\overline{S}$ of $\overline{\Phi}$ as desired.
\end{proof}
We need one last definition to prove that certain Gabor-Steiner ETFs correspond to roux lines.
\begin{defn}
Let $A$ be a matrix. The \emph{$N$th Hadamard product} $A^{\circ N}$ of $A$ is the $N$th component-wise product.  Namely, $(A^{\circ N})_{j,k} = (A_{j,k})^N$.
\end{defn}
We may now present the so-called roux lines detector~\cite[Corollary 4.6]{IvMi18}.
\begin{prop}\label{prop:roux}
Given a normalized signature matrix $\overline{S}$, $\overline{S}$ corresponds to equal-norm representatives of roux lines if and only if the following occur simultaneously:
\begin{enumerate}
\item The entries of $\overline{S}$ are all roots of unity.
\item Every Hadamard power of $\overline{S}$ has exactly two eigenvalues.
\end{enumerate}
\end{prop}
\begin{thm}
Let $p\geq 3$ be prime.  For all $m=(p,p,\hdots,p)$, $\cG(m)$ is a set of roux lines.
\end{thm}
\begin{proof}
We let $\overline{S}$ be the normalized signature matrix of $\cG(m)$ presented in Proposition~\ref{prop:sign} and $s+1$ be the length of $m$. That (1) from Proposition~\ref{prop:roux} holds for $S$ is clear. Further, since $\zeta_p^N$ is a primitive $p$th root of unity for all $N$ such that $p\not\vert N$, such $N$th Hadamard powers of $\overline{S}$ simply yield normalized signature matrices of Gabor-Steiner equiangular tight frames generated by possibly different primitive $p$th roots of unity.  These $\overline{S}^{\circ N}$ all have two eigenvalues.  If $p \vert N$, then $\overline{S}^{\circ N}$ has negative ones in every entry that is neither on the diagonal nor the first row or column.  Such a $\overline{S}^{\circ N}$ is a normalized signature matrix for a simplex of $p^{2s+2}$ vectors spanning a $p^{2s+1}$-dimensional space and thus also has two eigenvalues.
\end{proof}
Thus the Gabor-Steiner ETF generated from any finite, abelian $p$-group is roux. We note that an immediate porism of this result is that the so-called Naimark complement~~\cite{frame_book,Waldron18} of any $\cG(p, \hdots, p)$ with $p$ odd prime yields a cyclic DRACKN (distance-regular cover of the complete graph whose automorphism group that fixes each fibre as a set is cyclic) \cite{CGSZ16}.  We thank Joey Iverson for pointing out this result concerning DRACKNs.



\end{document}